\documentclass{amsart}

\usepackage[hyperindex]{hyperref}
\usepackage{color}

\newcommand{\balg}{\begin{algorithm}}
\newcommand{\ealg}{\end{algorithm}}
\newcommand{\br}{\begin{remark}}
\newcommand{\er}{\end{remark}}
\newcommand{\bex}{\begin{example}}
\newcommand{\eex}{\end{example}}

\newtheorem{theorem}{Theorem}[section]

\theoremstyle{definition}

\newtheorem{example}[theorem]{Example}

\newtheorem{algorithm}[theorem]{Algorithm}

\newtheorem{remark}[theorem]{Remark}

\numberwithin{equation}{section}

\begin{document}

\title{Completely positive tensor decomposition}



\author{Jinyan Fan}
\address{Department of Mathematics, and MOE-LSC, Shanghai Jiao Tong University,
Shanghai 200240, P.R. China}
\email{jyfan@sjtu.edu.cn}

%

\author{Anwa Zhou}
\address{Department of Mathematics, Shanghai Jiao Tong University,
Shanghai 200240, P.R. China}
\email{congcongyan@sjtu.edu.cn}
\thanks{}

\thanks{The first author is partially supported by NSFC 11171217.}


\date{}

\dedicatory{}

\begin{abstract}
A symmetric  tensor, which has a symmetric nonnegative decomposition,
is called a completely positive tensor.
We consider the completely positive tensor decomposition problem.
A semidefinite algorithm is presented for checking whether a  symmetric tensor is completely positive.
If it is not completely positive, a certificate for it can be obtained;
if it is completely positive, a nonnegative decomposition can be obtained.
\end{abstract}


\maketitle

\section{Introduction}

Let $m$ and $n$ be positive integers. An  $m$th order $n$-dimensional tensor $\mathcal{A}$ is
an array indexed by integer tuples $(i_1, \ldots , i_m)$ with $1 \leq i_1, \ldots, i_m \leq n$, i.e.,
\[
\mathcal{A}=(\mathcal{A}_{i_1,\ldots,i_m})_{1 \leq i_1, \ldots , i_m \leq n}.
\]
Let $\mathrm{T}^m(\mathrm{R}^{n})$ be the set of all such real tensors.
A tensor $\mathcal{A} \in \mathrm{T}^m(\mathrm{R}^{n})$
is symmetric if each entry $\mathcal{A}_{i_1,\ldots,i_m}$ is invariant with respect to all permutations of $(i_1, \ldots , i_m)$.
Let $\mathrm{S}^m(\mathrm{R}^{n})$ be the set of all symmetric tensors in $\mathrm{T}^m(\mathrm{R}^n)$.
For a vector $v \in\mathrm{R}^n$, we denote $v^{\otimes m}$ the $m$th order $n$-dimensional symmetric
 outer product tensor such that
\[
(v^{\otimes m})_{i_1,\ldots,i_m}= v_{i_1}\cdots v_{i_m}.
\]
$v^{\otimes m} $ is a rank-1 symmetric tensor.
Comon et al. \cite{Comon2008} showed that any real symmetric tensor is a linear combination of rank-1 symmetric tensors.
Interested readers are referred to \cite{Comon2002,Cui,Hu, Kolda2009,Kolda2011,Nie2014a, Nie2014b} for numerical methods for computing real eigenvalue of symmetric tensors, tensor decompositions
as well as applications.

Let $\mathrm{R}^n_+$ be the nonnegative orthant.
A symmetric tensor $\mathcal{A}\in \mathrm{S}^m(\mathrm{R}^{n})$ is completely positive,
if there exit nonnegative vectors $v^1,\ldots,v^r\in \mathrm{R}^n_+$ such that
\begin{equation}\label{cptensorf}
  \mathcal{A} = \sum^r_{k=1} (v^k)^{\otimes m},
\end{equation}
where $r$ is called the length of the decomposition \eqref{cptensorf} (cf. \cite{Qi}).
The smallest $r$ in the above is called the CP-rank of $\mathcal{A}$.
If $\mathcal{A}$ is completely positive, \eqref{cptensorf} is called a nonnegative (or completely positive) decomposition of $\mathcal{A}$.
The completely positive tensor is an extension of the completely positive matrix \cite{BermanN,Burer,Xu,ZhouFan13,ZhouFan14}.

For $\mathcal{B}\in \mathrm{S}^m(\mathrm{R}^{n})$, we define
 \[
 \mathcal{B}x^m := \sum_{1\leq i_1,\ldots,i_m\leq n} \mathcal{B}_{i_1,\ldots,i_m} x_{i_1}\cdots x_{i_m}.
 \]
If
$$
\mathcal{B}x^m \geq 0, \quad \forall x \in \mathrm{R}^n_+,
$$
we call $\mathcal B$ a copositive tensor (cf. \cite{Qi13}).
The copositive tensor is an extension of the copositive matrix.
Obviously, both symmetric nonnegative tensors and positive semidefinite tensors are copositive tensors.

For $\mathcal{A}, \mathcal{B}\in  \mathrm{S}^m(\mathrm{R}^{n})$,
the inner product of $\mathcal{A}$ and $\mathcal{B}$ is defined as
$$
\mathcal{A}\bullet \mathcal{B}=\sum_{1\leq i_1,\ldots, i_m\leq n}
\mathcal{A}_{i_1,\ldots, i_m} \mathcal{B}_{i_1,\ldots, i_m}.
$$
For a cone $ C\subseteq \mathrm{S}^m(\mathrm{R}^{n})$, the dual cone of $ C$ is defined as
$$
 C^*:= \{\mathcal{G} \in \mathrm{S}^m(\mathrm{R}^{n}) : \mathcal{A}\bullet \mathcal{G} \geq 0 \; \text{for all} \; \mathcal{A} \in  C\}.
$$
Denote by $CP_{m,n}$ and $COP_{m,n}$ the sets of $m$th order $n$-dimensional completely positive tensors and copositive tensors, respectively.
Both $CP_{m,n}$ and $COP_{m,n}$ are proper cones,
i.e., they are closed pointed convex cones with non-empty interiors \cite{Pena14}.
Moreover, they are dual to each other \cite{Qi}.

Completely positive tensor and decomposition have wide applications in statistics, computer vision, exploratory multiway data analysis and blind source separation \cite{Cichocki, Shashua}.
Qi et al. \cite{Qi} showed that a strongly symmetric hierarchically dominated nonnegative tensor is completely positive and presented a hierarchical elimination algorithm for checking this.
Kolda \cite{Kolda2014} formulated the completely positive tensor decomposition problem as a nonnegative constrained least squares problem by assuming the length of the decomposition is known, then solved it by SNOPT \cite{Gill2005, Gill2008}.

A natural and interesting problem is: how to check whether a general symmetric tensor is completely positive or not?
If it is not completely positive, how can we get a certificate for it?
If it is completely positive, how can we get a nonnegative decomposition for it?
As we know, checking a completely positive matrix is NP-hard \cite{Dickinson11}.
The problems of checking a completely positive tensor and giving a completely positive tensor decomposition are more complicated. We solve them in this paper.

The paper is organized as follows. In Section 2, we characterize when a symmetric tensor is completely positive.
In Section 3, we formulate the completely positive tensor decomposition problem as an $E$-truncated $K$-moment problem and present a semidefinite algorithm for it.
The convergence properties of the algorithm are also studied.
Some numerical examples are given in Section 4.
Finally, we conclude the paper with some discussions in section 5.

\section{Characterization as moments}
In this section, we characterize the completely positive tensor as a truncated moment sequence,
and give some necessary and sufficient conditions for a symmetric tensor to be completely positive.

As we know, a symmetric matrix can be identified by a vector that consists of its upper triangular entries. Similarly, a symmetric tensor $\mathcal{A}\in  \mathrm{S}^m(\mathrm{R}^{n})$ can also be identified by a vector that
consists of its upper triangular entries, i.e. the entries
$$
\mathcal{A}_{i_1,\ldots,i_m}\ \mbox{with}\ 1 \leq i_1\leq \ldots\leq i_m \leq n.
$$

Let $\mathrm{N}$ be the set of nonnegative integers.
For $\alpha = (\alpha_1,\cdots, \alpha_n) \in \mathrm{N}^n$,
denote $|\alpha| := \alpha_1+\cdots+\alpha_n$.
Let
\begin{equation}\label{AE} E := \{\alpha
\in \mathrm{N}^n:\, |\alpha|=m\}.
\end{equation}
Then, each index $(i_1,\ldots i_m)$ corresponds to a vector
$$
e_{i_1}+e_{i_2}+\cdots+e_{i_m}\in E,
$$
where $e_i$ is the $i$-th unit vector in $\mathrm{R}^n$.
Hence, the identifying vector of $\mathcal{A}\in \mathrm{S}^m(\mathrm{R}^{n})$ can also be written as
$$
\mathbf{a} =(\mathbf{a}_{\alpha})_{\alpha\in E} \in \mathrm{R}^E,
$$
where $\mathrm{R}^{E}$ denotes the
space of real vectors indexed by $\alpha \in E$.
We call $\mathbf{a}$ an $E$-truncated moment sequence ($E$-tms).

Let
\begin{equation} \label{KE}
K=\{x\in \mathrm{R}^n :\,
x^T x -1= 0, x_1 \geq0, x_2 \geq 0, \cdots, x_n\geq 0\}.
\end{equation}
Note that every nonnegative vector is a multiple
of a vector in  $K$.
So, by (\ref{cptensorf}),
$\mathcal{A}\in CP_{m,n}$ if and only if
there exist   $\rho_1,\cdots,\rho_r >0$
and $u^1,\cdots,u^r \in K$ such that
 \begin{equation}\label{Ecompletely positivee}
\mathcal{A}=\rho_1 (u^1)^{\otimes m} +\cdots +\rho_r (u^r)^{\otimes m}.
 \end{equation}

 The $E$-truncated $K$-moment problem
($E$-T$K$MP) studies whether or not a given {$E$-tms $\mathbf{a}$ admits a
$K$-measure} $\mu$, i.e., a nonnegative Borel measure $\mu$ supported
in $K$ such that
\[
\mathbf{a}_{\alpha}= \int_K x^{\alpha} d \mu, \quad
\forall\, \alpha \in  E,
\]
where $x^{\alpha} := x^{\alpha_1}_1 \cdots x^{\alpha_n}_n$.
A measure $\mu$ satisfying the above is called a
$K$-representing measure for $\mathbf{a}$. A measure is called finitely
atomic if its support is a finite set, and is called $r$-atomic if its
support consists of at most $r$ distinct points.
We refer to \cite{Nie} for representing measures of truncated moment sequences.

Hence, by \eqref{cptensorf}, a tensor $\mathcal{A}\in \mathrm{S}^m(\mathrm{R}^{n})$, with the identifying vector
$\mathbf{a} \in \mathrm{R}^{E}$,  is completely positive if and only if $\mathbf{a}$ admits
an $r$-atomic $K$-measure,  i.e.,
\begin{equation}
\label{Ecompletely positiveee} \mathbf{a}=\rho_1 [u^1]_{E} +\cdots +\rho_r [u^r]_{E},
\end{equation}
where each $\rho_i>0$, $u^i \in K$  and
\[
[u]_{E} :=(u^{\alpha})_{\alpha \in
{E}}.
\]
In other words,
checking a completely positive tensor is equivalent to an $E$-T$K$MP with
$E$ and $K$ given in (\ref{AE}) and (\ref{KE}) respectively.
Denote
\[
Q=\{\mathbf{a}\in \mathrm{R}^E:  \mathbf{a}\ \mbox{admits a}\ K\mbox{-measure} \}.
\]
Then, $Q$ is the cone of completely positive tensor. So, we have
\begin{align}\label{completely positiveE1}
\mathcal{A}\in CP_{m,n} \quad \mbox{if and only if}\quad  \mathbf{a} \in  Q.
\end{align}

Denote
\[
\mathrm{R}[x]_{E}:= \mbox{span}\{x^{\alpha}: \alpha\in E\}.
\]
For a polynomial $p\in \mathrm{R}[x]_{E}$, the notion $p|_K\geq 0$ (resp., $p|_K> 0$) means
$p\geq 0$ (resp., $p> 0$) on $K$. The dual cone of $Q$ is
$$
P=\{p\in \mathrm{R}[x]_{E}: p|_K\geq 0\},
$$
which is $COP_{m,n}$, the cone of copositive tensors (cf. \cite{Laurent2009, Nie1}).
We say $\mathrm{R}[x]_{E}$ is {$K$-full} if there exists a polynomial $p \in \mathrm{R}[x]_{E}$ such
that $p|_{K} >0$ (cf. \cite{Nie4}).
In fact, if we choose $p(x)=\sum_{i=1}^n x_i^m\in \mathrm{R}[x]_{E}$,
then $p|_{K} >0$.
So, $\mathrm{R}[x]_{E}$ is $K$-full.

An $E$-tms $\mathbf{a} \in \mathrm{R}^{E}$ defines a Riesz function $F_{\mathbf{a}}$ acting on
$\mathrm{R}[x]_{E}$ as
 \begin{equation}\label{La} F_{\mathbf{a}}
(\sum_{\alpha\in E}p_{\alpha} x^{\alpha}):=
\sum_{\alpha\in E}p_{\alpha} \mathbf{a}_{\alpha}.
 \end{equation}
 We also denote $\langle p,\mathbf{a} \rangle:= F_{\mathbf{a}}(p)$ for convenience.
Let
$$
\mathrm{N}_d^n := \{ \alpha \in \mathrm{N}^n: \, |\alpha| \leq d\}
\quad \mbox{and} \quad
\mathrm{R}[x]_{d}:=\mbox{span}\{x^{\alpha}: \alpha\in \mathrm{N}^n_{d}\}.
$$
For $s \in \mathrm{R}^{\mathrm{N}^n_{2k}}$ and $q
\in \mathrm{R}[x]_{2k}$, the $k$-th localizing matrix  of
$q$ generated by $s$ is the symmetric
matrix $L^{(k)}_q (s)$ satisfying
 \begin{equation}\label{Lzqp2}
F_s (q p^2)= \operatorname{vec}(p)^T (L^{(k)}_q (s))\operatorname{vec}(p), \quad \forall p\in
\mathrm{R}[x]_{k-\lceil deg(q)/2 \rceil},
 \end{equation}
where $\operatorname{vec}(p)$ denotes the coefficient vector of $p$ in the
graded lexicographical ordering, and $\lceil t\rceil$
denotes the smallest integer that is not smaller than $t$.
In particular, when $q=1$, $L^{(k)}_1 (s)$ is called a  $k$-th
order moment matrix  and denoted as $M_k(s)$.
In fact, we have
$$
L^{(k)}_q (s) = (\sum_{\alpha} q_{\alpha} s_{\alpha+\beta+\gamma})_{\beta,\gamma\in \mathrm{N}^n_{k-\lceil deg(q)/2 \rceil}},
$$
and
$$
M_k(s) = L^{(k)}_1 (s) = (s_{\beta+\gamma})_{\beta,\gamma \in \mathrm{N}^n_{k}}.
$$
We refer to \cite{Lasserre2001,Lasserre2008, Lasserre2009, Laurent2009, Nie} for more details about localizing matrices
and moment matrices.

Denote the polynomials:
\[
h(x):= x^T x -1, \, g_0(x):=1,\, g_1(x): =x_1, \, \ldots,  g_{n}(x): = x_n.
\]
Note that $K$ given in (\ref{KE}) is nonempty compact.
 It can also be described equivalently as
\begin{equation}\label{K2}
K=\{x\in \mathrm{R}^n: \ h(x)= 0, g(x) \geq 0\},
\end{equation}
where $g(x)=(g_0(x), g_1(x),\cdots,g_{n}(x))$.
For example, when $n = 2$ and $k = 2$,
the second order localizing matrices of the above polynomials are:
\begin{align*}
 M_2(s) :=  L^{(2)}_{1} (s)= \left(\begin{array}{cccccc}
  s_{0,0} &s_{1,0} &s_{0,1} &s_{2,0} &s_{1,1} &s_{0,2}\\
  s_{1,0} &s_{2,0} &s_{1,1} &s_{3,0} &s_{2,1} &s_{1,2}\\
  s_{0,1} &s_{1,1} &s_{0,2} &s_{2,1} &s_{1,2} &s_{0,3}\\
  s_{2,0} &s_{3,0} &s_{2,1} &s_{4,0} &s_{3,1} &s_{2,2}\\
  s_{1,1} &s_{2,1} &s_{1,2} &s_{3,1} &s_{2,2} &s_{1,3}\\
  s_{0,2} &s_{1,2} &s_{0,3} &s_{2,2} &s_{1,3} &s_{0,4}
\end{array}\right),
 \end{align*}
  \begin{align*}
 L^{(2)}_{x_1} (s)&=\left(\begin{array}{ccc}
  s_{1,0} &s_{2,0} &s_{1,1}\\
  s_{2,0} &s_{3,0} &s_{2,1} \\
  s_{1,1} &s_{2,1} &s_{1,2}
\end{array}\right), \quad
 L^{(2)}_{x_2} (s)= \left(\begin{array}{ccc}
  s_{0,1} &s_{1,1} &s_{0,2}\\
  s_{1,1} &s_{2,1} &s_{1,2} \\
  s_{0,2} &s_{1,2} &s_{0,3}
\end{array}\right),
\\
 L^{(2)}_{x_1^2+x_2^2-1} (s)&= \left(\begin{array}{ccc}
 s_{2,0}+s_{0,2}-s_{0,0} &s_{3,0}+s_{1,2}-s_{1,0} &s_{2,1}+s_{0,3}-s_{0,1}\\
 s_{3,0}+s_{1,2}-s_{1,0} &s_{4,0}+s_{2,2}-s_{2,0} &s_{3,1}+s_{1,3}-s_{1,1} \\
 s_{2,1}+s_{0,3}-s_{0,1} &s_{3,1}+s_{1,3}-s_{1,1} &s_{2,2}+s_{0,4}-s_{0,2}
\end{array}\right).
\end{align*}

As shown in \cite{Nie}, a necessary condition for $s \in \mathrm{R}^{\mathrm{N}^n_{2k}}$
to admit a $K$-measure is
\begin{equation}
\label{SDPC}
  L^{(k)}_{h} (s) = 0, \quad \mbox{and}\quad L^{(k)}_{g_j} (s) \succeq
0, \quad j=0,1,\cdots,n.
\end{equation}
If, in addition to (\ref{SDPC}), $s$ satisfies the {rank condition}
\begin{equation}
\label{RC}
\text{rank} (M_{k-1}(s)) =\text{rank} (M_{k} (s)),
\end{equation}
then $s$ admits a unique
$K$-measure, which is $\text{rank} M_k(s)$-atomic
(cf. Curto and Fialkow \cite{CurtoF}).
We say that $s$ is {flat} if both (\ref{SDPC}) and (\ref{RC}) are satisfied.

Given two tms' $y \in \mathrm{R}^{\mathrm{N}^n_{d}}$ and $z \in
\mathrm{R}^{\mathrm{N}^n_{e}}$, we say $z$ is an  extension  of $y$, if $d\leq e$ and
$y_{\alpha} = z_{\alpha}$ for all $\alpha \in \mathrm{N}^n_{d}$. We denote
by $z|_{E}$ the subvector of $z$, whose entries are indexed by
$\alpha \in E$. For convenience, we denote by $z|_{d}$
the subvector $z |_{\mathrm{N}^n_{d}}$.
If $z$ is flat and extends $y$, we say $z$ is a {flat extension} of $y$.
It is shown in \cite{Nie} that an $E$-tms $\mathbf{a} \in \mathrm{R}^{E}$ admits a $K$-measure
if and only if it is extendable to a flat tms $z \in \mathrm{R}^{\mathrm{N}^n_{2k}}$
for some $k$.
So, by (\ref{completely positiveE1}),
\begin{align}\label{completely positiveE4}
\mathcal{A}\in CP_{m,n} \quad \mbox{if and only if}\quad  \mathbf{a} \ \mbox{has a flat extension}.
\end{align}
Therefore,  checking whether a symmetric tensor $\mathcal{A}\in \mathrm{S}^m(\mathrm{R}^{n})$ is completely positive is equivalent to checking whether its identifying vector $\mathbf{a} \in \mathrm{R}^{E}$ has a flat extension.

For example, consider the symmetric tensor $\mathcal A\in \mathrm{S}^3(\mathrm{R}^{3})$ given as
\[
\mathcal A(:,:,1) =\left(
                     \begin{array}{ccc}
                       2  &   1  &   1 \\
	                   1  &   1  &   0 \\
	                   1  &   0  &   1 \\
                     \end{array}
                   \right),
\mathcal A(:,:,2) =\left(
                     \begin{array}{ccc}
                       1   &  1  &   0 \\
	                   1   &  2  &   0 \\
	                   0   &  0  &   0 \\
                     \end{array}
                   \right),
\mathcal A(:,:,3) =\left(
                     \begin{array}{ccc}
                       1   &  0 &    1 \\
	                   0   &  0 &    0 \\
	                   1   &  0 &    1 \\
                     \end{array}
                   \right).
\]
The index set of $\mathcal{A}$ is  $E := \{\alpha \in \mathrm{N}^3:\, |\alpha|=3\}$ and the identifying vector of $\mathcal A$ is
\[\mathbf{a} =(2, 1, 1, 1, 0, 1, 2, 0, 0, 1)^T\in \mathrm{R}^{E}.\]
For $n=3$ and $k=2$, consider the tms
\begin{align*}
 \tilde{a}=(& 6.6569,    4.0000,    3.0000,    2.0000,    2.8284,    1.4142,    1.4142,    2.4142,    0.0000,    1.4142, \\
                  & 2.0000,    1.0000,    1.0000,    1.0000,         0.0000,    1.0000, 2.0000,         0.0000,         0.0000, 1.0000, \\
                  & 1.4142,   0.7071,    0.7071,    0.7071,   0.0000,    0.7071,    0.7071,   0.0000,    0.0000,   0.7071, \\
                  & 1.7071,   0.0000,   0.0000, 0.0000,  0.7071)^T\in \mathrm{R}^{\mathrm{N}^3_{2k}}.
\end{align*}
We compute the second order localizing matrices of $h(x)$ and $g(x)$ generated by $\tilde a$
and obtain:
\begin{align*}
 L^{(2)}_{x_1^2+x_2^2+x_3^2-1} (\tilde{a})= \left(\begin{array}{cccc}
  0.0000 &0.0000 &0.0000 &0.0000\\
  0.0000 &0.0000 &0.0000 &0.0000 \\
  0.0000 &0.0000 &0.0000 &0.0000 \\
  0.0000 &0.0000 &0.0000 &0.0000
\end{array}\right)=0,
\\
 L^{(2)}_{x_1} (\tilde{a})=\left(\begin{array}{cccc}
  4.0000 & 2.8284 & 1.4142 & 1.4142\\
  2.8284 & 2.0000 & 1.0000 & 1.0000 \\
  1.4142 & 1.0000 & 1.0000 & 0.0000 \\
  1.4142 & 1.0000 & 0.0000 & 1.0000
\end{array}\right)\succeq 0,
\\
 L^{(2)}_{x_2} (\tilde{a})=\left(\begin{array}{cccc}
  3.0000 &1.4142 &2.4142 &0.0000\\
  1.4142 &1.0000 &1.0000 &0.0000 \\
  2.4142 &1.0000 &2.0000 &0.0000 \\
  0.0000 &0.0000 &0.0000 &0.0000
\end{array}\right)\succeq 0,
\\
L^{(2)}_{x_3} (\tilde{a})=\left(\begin{array}{cccc}
  2.0000 &1.4142 &0.0000 &1.4142\\
  1.4142 &1.0000 &0.0000 &1.0000 \\
  0.0000 &0.0000 &0.0000 &0.0000 \\
  1.4142 &1.0000 &0.0000 &1.0000
\end{array}\right)\succeq 0,
\end{align*}
 \begin{align*}
&L^{(2)}_{1} (\tilde{a})=M_2(\tilde{a})\\
&={\scriptsize\left(
  \begin{array}{cccccccccc}
    6.6569 &   4.0000  &  3.0000  &  2.0000 &   2.8284  &  1.4142  &  1.4142  &  2.4142  &  0.0000  &  1.4142\\
    4.0000  &  2.8284   & 1.4142  &  1.4142  &  2.0000 &   1.0000  &  1.0000  &  1.0000  &  0.0000  &  1.0000\\
    3.0000  &  1.4142  &  2.4142  &  0.0000  &  1.0000 &   1.0000  &  0.0000  &  2.0000  &  0.0000  &  0.0000\\
    2.0000  &  1.4142  &  0.0000   & 1.4142  &  1.0000 &   0.0000  &  1.0000 &   0.0000  &  0.0000  &  1.0000\\
    2.8284  &  2.0000  &  1.0000   & 1.0000  &  1.4142  &  0.7071  &  0.7071  &  0.7071  &  0.0000   & 0.7071\\
    1.4142  &  1.0000  &  1.0000   & 0.0000  &  0.7071  &  0.7071  &  0.0000  &  0.7071  &  0.0000 &   0.0000\\
    1.4142  &  1.0000  &  0.0000   & 1.0000  &  0.7071  &  0.0000  &  0.7071  &  0.0000 &   0.0000 &   0.7071\\
    2.4142  &  1.0000  &  2.0000   & 0.0000  &  0.7071  &  0.7071  &  0.0000 &   1.7071 &   0.0000 &   0.0000\\
    0.0000  &  0.0000  &  0.0000   & 0.0000 &   0.0000  &  0.0000  &  0.0000  &  0.0000  &  0.0000  &  0.0000\\
    1.4142  &  1.0000   & 0.0000   & 1.0000  &  0.7071  &  0.0000  &  0.7071 &   0.0000  &  0.0000  &  0.7071
  \end{array}
\right)}\\
& \succeq 0.
\end{align*}
Since the first order moment matrix is
$$
M_1 (\tilde{a})=\left(\begin{array}{cccc}
  6.6569 &   4.0000  &  3.0000  &  2.0000\\
  4.0000  &  2.8284   & 1.4142  &  1.4142 \\
  3.0000  &  1.4142  &  2.4142  &  0.0000 \\
  2.0000  &  1.4142  &  0.0000   & 1.4142
\end{array}\right),
$$
we have
$$
\operatorname{rank}(M_1 (\tilde{a}))=\operatorname{rank}(M_2(\tilde{a}))=3.
$$
Thus, by \eqref{SDPC} and \eqref{RC},  $\tilde{a}$ is fat.
It is easy to check that $\tilde{a}|_{E}=\mathbf{a}$,
which implies that $\tilde{a}$ is a flat extension of $\mathbf{a}$.
So, by \eqref{completely positiveE1}, $\mathcal A$ is completely positive.

\section{A semidefinite algorithm}
In this section, we present a semidefinite algorithm for
checking whether a symmetric tensor is completely positive or not.
If it is not completely positive, we can give a certificate for it;
if it is completely positive, we can give a nonnegative decomposition of it.

As analysed  in Section 2, we transform the problem of checking the complete positivity of a symmetric tensor to checking the existence of a flat extension of its identifying vector.
Let $d > m$ be an even integer. Choose a polynomial
$F \in \mathrm{R}[x]_d$ :
\[
F(x)  = \sum\limits_{\alpha\in \mathrm{N}^n_{d}}F_{\alpha} x^{\alpha}.
\]
Consider the linear optimization problem
\begin{equation}
\label{MPP}
  \begin{array}{cl}
  \min\limits_{z}  & \langle F, z \rangle\\\
  \mbox{s.t.} & z|_{E} = \mathbf{a}, z \in \Upsilon_d,
 \end{array}
\end{equation}
where
$$
\Upsilon_d=\{z\in \mathrm{R}^{\mathrm{N}_d^n}:  z\ \mbox{admits a}\ K\mbox{-measure} \}.
$$
Since $K$ is a compact set and $\mathrm{R}[x]_{E}$ is $K$-full,
  the feasible set of (\ref{MPP}) is compact convex. So,  (\ref{MPP}) has a minimizer for all $F$.
Usually, we choose $F \in \Sigma_{n,d}$, where $\Sigma_{n,d}$ is
the set of all sum of squares polynomials in $n$ variables with degree $d$.
As $\Upsilon_d$ is difficult to describe, by \eqref{SDPC}, we relax it by the cone
\begin{equation}
\label{EIk}
  \Gamma_{k} (h,g) := \left\{z \in \mathrm{R}^{\mathrm{N}^n_{2k}}
  :  L^{(k)}_{h} (z) = 0, L^{(k)}_{g_j} (z) \succeq 0,
  j=0,1,\ldots,n
  \right\},
\end{equation}
with $k \geq d/2$ an integer. The {$k$-th order semidefinite
relaxation} of (\ref{MPP}) is
\begin{equation}
\label{SDPR}
   \begin{array}{cl}
  \min\limits_{z}  & \langle F, z \rangle\\
  \mbox{s.t.} & z|_{E} = \mathbf{a}, z \in \Gamma_k(h,g).
 \end{array}
\end{equation}
Based on solving the hierarchy of (\ref{SDPR}),
we present a semidefinite algorithm for checking completely positive tensors as follows.

\balg\label{Algorithm}  A semidefinite algorithm for checking completely positive tensors.

\textbf{Step 0.} Choose a generic $F \in \Sigma_{n,d}$, and let $k := d/2$.

\textbf{Step 1.} Solve (\ref{SDPR}). If (\ref{SDPR}) is infeasible, then
$\mathbf{a}$ doesn't admit a $K$-measure, and stop.
 Otherwise, compute a minimizer $z^{*,k}$. Let $t :=1$.

\textbf{Step 2.} Let $w := z^{*,k}|_{2t}$. If the rank condition (\ref{RC})
is not satisfied, go to Step 4.

\textbf{Step 3.} Compute the finitely atomic measure $\mu$ admitted by $w$:
$$\mu = \rho_1 \delta(u^1) + \cdots + \rho_r \delta(u^r),$$
where $\rho_i > 0$, $u^i \in K$, $r= \text{rank} M_t (w)$, and $\delta(u^i)$ is the Dirac
measure supported on the point $u^i$ $(i=1,\cdots,r)$. Stop.

\textbf{Step 4.} If $t < k$, set $t := t+1$ and go to Step 2; otherwise, set
$k := k+1$ and go to Step 1. \ealg

Denote $[x]_{d} := (x^{\alpha})_{\alpha \in \mathrm{N}^n_{d}}$.
We choose $F = [x]^T_{d/2} J^T J [x]_{d/2}$ in (\ref{SDPR}), where $J$ is a random square matrix
obeying Gaussian distribution. We check the rank condition
(\ref{RC}) numerically with singular value decomposition \cite{Golub}. The rank of a matrix is evaluated as the
number of its singular values that are greater than or equal to
$10^{-6}$. We use the method given in \cite{HenrionJ} to get a $r$-atomic
$K$-measure for $w$.

\begin{theorem}
\label{Algorithmresults}
Algorithm \ref{Algorithm} has the following properties:
\begin{itemize}
\item [(1)] If (\ref{SDPR}) is infeasible for some $k$,
then $\mathbf{a}$ admits no $K$-measures and the
corresponding tensor $\mathcal{A}$ is not completely positive.

\item [(2)] If the tensor $\mathcal{A}$ is not completely positive,
then (\ref{SDPR}) is infeasible for all $k$ big enough.

\item [(3)] If the tensor $\mathcal{A}$ is completely positive, then
for almost all generated $F$,
we can asymptotically get a flat extension of $\mathbf{a}$
by solving the hierarchy of (\ref{SDPR}).

\end{itemize}
\end{theorem}

\begin{proof}
Since $\mathrm{R}[x]_{E}$ is $K$-full for the $E$ and $K$ given in \eqref{AE} and \eqref{KE} respectively,
we can deduce the convergence results from \cite[Section 5]{Nie}.
\end{proof}

If the tensor $\mathcal{A}$ is completely positive,
under some general conditions, which is almost necessary and sufficient,
we can get a flat extension of $\mathbf{a}$ by solving the hierarchy of (\ref{SDPR}),
within finitely many steps (cf.~\cite{Nie,Nie3}). This always happens in our
numerical experiments.
After getting a flat extension of $\mathbf{a}$, we can get a $r$-atomic
$K$-measure for $\mathbf{a}$, which then produces a nonnegative decomposition of $\mathcal{A}$.

\section{Numerical experiments}

In this section, we present some numerical experiments
for checking whether a symmetric tensor is completely positive by  Algorithm \ref{Algorithm}.
The nonnegative decomposition of the tensor is also given if it is completely positive.
We use softwares GloptiPoly~3 \cite{HenrionJJ} and SeDuMi \cite{Sturm}
to solve the semidefinite relaxation problems (\ref{SDPR}).

\bex\label{Example2} \upshape
Consider the tensor $\mathcal{A}\in \mathrm{S}^3(\mathrm{R}^{11})$ given as:
\begin{equation}\label{Exa2}
\mathcal{A} = \sum^6_{k=1}
(u^k)^{\otimes 3},
\end{equation}
where $u^k$ is randomly generated:
\begin{align*}
  ( u^1 ,  u^2 , & u^3 ,  u^4 ,  u^5 ,  u^6)\\
=& \left(
  \begin{array}{cccccc}
    1.6264  &  1.5915    &     0  & -0.3974  &  1.6823   &      0 \\
         0   & 1.0873  & -0.2827  & -0.1396  &  0.4450   &      0 \\
    1.1821  &  0.9580 &   0.5106 &   0.0372  &       0   & 1.3210 \\
    0.5190  & -0.0254  & -0.4710  &       0  &  0.5333  & -0.4526 \\
    0.1157  &       0  & -0.4448 &  -0.0253   & 1.2527  &  1.0004 \\
   -0.2380  & -0.1910  &  1.4544  &       0   & 1.1283  &  0.4371 \\
   -0.1857  &  0.7614   &-0.1338  &       0   & 0.1025  &       0 \\
    0.8829  &  0.6716  &       0  & -0.3942  &       0  &  0.0551 \\
         0  &  1.6044  &       0  &  0.2774  &       0  &  0.6739 \\
         0  &  1.2222  &  0.2452  &  0.2778 &  -0.3661   &      0 \\
    3.0903   & 0.4538  & -0.4791  &  1.2668 &   0.5206   &      0 \\
  \end{array}
\right).
\end{align*}

We apply Algorithm \ref{Algorithm} and choose $d=4$ and $k=2$ in Step 0.
It terminates at Step 1 with $k=2$, as (\ref{SDPR}) is infeasible.
So, $\mathcal{A}$ is not completely positive.

\eex

\bex\label{Example2inf} \upshape
Consider the tensor $\mathcal{A}\in \mathrm{S}^5(\mathrm{R}^{8})$ given as:
\begin{equation}\label{Exa2inf}
\mathcal{A} = \sum^6_{k=1}
(u^k)^{\otimes 5},
\end{equation}
where $u^k$ is randomly generated:
\begin{align*}
  ( u^1 ,  u^2 , & u^3 ,  u^4 ,  u^5 ,  u^6)\\
=&   \left(
  \begin{array}{cccccc}
   -0.2758  &  1.0310  &  0.6933 &   1.3952 &   0.8384  & -0.5601 \\
    0.1657 &  -0.1714  &  0.3189 &  -0.4789 &  -0.6165  &  1.4245 \\
   -1.7971 &  -1.2471  & -0.0829  & -0.7241  & -0.7012  & -1.3621 \\
    0.1326 &  -1.0884  & -1.6137 &  -0.0201  & -1.0025  &  0.0263 \\
    0.0529   & 0.2557  & -0.5710 &   1.8326  & -0.8654  & -0.6023 \\
    0.2596  &  0.0034  &  0.5346  &  1.0681  & -2.0257  & -0.0331 \\
   -0.8960  & -0.8216  &  0.0039  & -1.1966 &  -0.7576 &  -0.2351 \\
   -1.5145 &  -0.2392  & -0.4547  &  0.3964  &  0.6182 &  -0.8991 \\
  \end{array}
\right).
\end{align*}

We apply Algorithm \ref{Algorithm}  and choose $d=6$ and $k=3$ in Step 0.
 It terminates at Step 1 with $k=3$, as (\ref{SDPR}) is infeasible.
So, $\mathcal{A}$ is not completely positive.

\eex

\bex\label{ExampleQi11} \upshape
Consider the tensor $\mathcal{A}\in \mathrm{S}^3(\mathrm{R}^{10})$ given as (cf. \cite{Qi}):
$$
\begin{array}{l}
  \mathcal{A}_{2,2,2}=4,\mathcal{A}_{2,2,3}=1,\mathcal{A}_{2,2,4}=1,\mathcal{A}_{2,2,5}=1,\mathcal{A}_{2,2,8}=1,\mathcal{A}_{2,3,3}=1,\\
  \mathcal{A}_{2,3,8}=1,\mathcal{A}_{2,4,4}=1,\mathcal{A}_{2,4,5}=1,\mathcal{A}_{2,5,5}=1,\mathcal{A}_{2,8,8}=1,\mathcal{A}_{3,3,3}=6,\\
  \mathcal{A}_{3,3,4}=1,\mathcal{A}_{3,3,5}=1,\mathcal{A}_{3,3,7}=1,\mathcal{A}_{3,3,8}=2,\mathcal{A}_{3,4,4}=1,\mathcal{A}_{3,4,5}=1,\\
  \mathcal{A}_{3,5,5}=1,\mathcal{A}_{3,7,7}=1,\mathcal{A}_{3,7,8}=1,\mathcal{A}_{3,8,8}=2,\mathcal{A}_{4,4,4}=7,\mathcal{A}_{4,4,5}=2,\\
  \mathcal{A}_{4,4,7}=1,\mathcal{A}_{4,4,9}=1,\mathcal{A}_{4,4,10}=1,\mathcal{A}_{4,5,5}=2,\mathcal{A}_{4,7,7}=1,\mathcal{A}_{4,7,9}=1, \\
  \mathcal{A}_{4,9,9}=1,\mathcal{A}_{4,10,10}=1,\mathcal{A}_{5,5,5}=4,\mathcal{A}_{7,7,7}=4,\mathcal{A}_{7,7,8}=1,\mathcal{A}_{7,7,9}=1, \\
  \mathcal{A}_{7,8,8}=1,\mathcal{A}_{7,9,9}=1,\mathcal{A}_{8,8,8}=6,\mathcal{A}_{8,8,9}=1,\mathcal{A}_{8,8,10}=1,\mathcal{A}_{8,9,9}=1, \\\mathcal{A}_{8,9,10}=1,
  \mathcal{A}_{8,10,10}=1,\mathcal{A}_{9,9,9}=4,\mathcal{A}_{9,9,10}=1,\mathcal{A}_{9,10,10}=1,\mathcal{A}_{10,10,10}=3,
\end{array}
$$
and the other entries are zero, except permutations of the above indices.
$\mathcal{A}$ is a strongly symmetric hierarchically dominated nonnegative tensor,
so it is completely positive \cite{Qi}.
A nonnegative decomposition of length 15 of $\mathcal{A}$ is also given in \cite{Qi}.

We apply Algorithm \ref{Algorithm} to verify that $\mathcal{A}$ is completely positive.
We choose $d=4$ and $k=2$ in Step 0. Algorithm \ref{Algorithm} terminates at Step 3 with $k=3$, and gives the following nonnegative decomposition
\[
\mathcal{A} =\sum_{i=1}^{14} \rho_i (v^i)^{\otimes 3},
\]
where $\rho_i$ and $v^i$ are:
{\footnotesize
\begin{align*}
 & \rho_1=2.0000,\ v^1=(0.0000,
    0.0000,
    0.0000,
    0.0000,
    1.0000,
    0.0000,
    0.0000,
    0.0000,
    0.0000,
    0.0000)^T, \\
 &\rho_2=2.6450,\ v^2=(0.0000,
    0.0000,
    0.0000,
    0.4169,
    0.0000,
    0.0000,
    0.0000,
    0.0000,
    0.0000,
    0.9090)^T,  \\
 &\rho_3=3.9427,\ v^3=(0.0000,
    0.0000,
    0.0000,
    0.9885,
    0.0000,
    0.0000,
    0.0000,
    0.0000,
    0.0000,
    0.1511)^T,  \\
 &\rho_4=2.0000,\ v^4=(0.0000,
    1.0000,
    0.0000,
    0.0000,
    0.0000,
    0.0000,
    0.0000,
    0.0000,
    0.0000,
    0.0000)^T,  \\
 &\rho_5=5.1962,\ v^5=(0.0000,
    0.5774,
    0.0000,
    0.5774,
    0.5774,
    0.0000,
    0.0000,
    0.0000,
    0.0000,
    0.0000)^T, \\
 &\rho_6=5.1962,\ v^6=(0.0000,
    0.0000,
    0.5774,
    0.5774,
    0.5774,
    0.0000,
    0.0000,
    0.0000,
    0.0000,
    0.0000)^T, \\
&\rho_7=2.0000,\ v^7=(0.0000,
    0.0000,
    0.0000,
    0.0000,
    0.0000,
    0.0000,
    1.0000,
    0.0000,
    0.0000,
    0.0000)^T, \\
 &\rho_8=3.0000,\ v^8=(0.0000,
    0.0000,
    1.0000,
    0.0000,
    0.0000,
    0.0000,
    0.0000,
    0.0000,
    0.0000,
    0.0000)^T, \\
 &\rho_9=3.0000,\ v^9=(0.0000,
    0.0000,
    0.0000,
    0.0000,
    0.0000,
    0.0000,
    0.0000,
    1.0000,
    0.0000,
    0.0000)^T, \\
 &\rho_{10}=2.0000,\ v^{10}=(0.0000,
    0.0000,
    0.0000,
    0.0000,
    0.0000,
    0.0000,
    0.0000,
    0.0000,
    1.0000,
    0.0000)^T, \\
 &\rho_{11}=5.1962,\ v^{11}=(0.0000,
    0.0000,
    0.0000,
    0.5773,
    0.0000,
    0.0000,
    0.5773,
    0.0000,
    0.5774,
    0.0000)^T, \\
 &\rho_{12}=5.1962,\ v^{12}=(0.0000,
    0.0000,
    0.0000,
    0.0000,
    0.0000,
    0.0000,
    0.0000,
    0.5773,
    0.5774,
    0.5773)^T, \\
 &\rho_{13}=5.1962,\ v^{13}=(0.0000,
    0.5774,
    0.5774,
    0.0000,
    0.0000,
    0.0000,
    0.0000,
    0.5774,
    0.0000,
    0.0000)^T, \\
 &\rho_{14}=5.1961,\ v^{14}=(0.0000,
    0.0000,
    0.5774,
    0.0000,
    0.0000,
    0.0000,
    0.5774,
    0.5774,
    0.0000,
    0.0000)^T.
\end{align*}
}

We get a shorter nonnegative decomposition of $\mathcal{A}$ than that given in \cite{Qi}. This shows an advantage of Algorithm \ref{Algorithm}.

\eex

\bex\label{ExampleQi2} \upshape
Consider the tensor $\mathcal{A}\in \mathrm{S}^4(\mathrm{R}^{10})$ given as (cf. \cite{Qi}):
$$
\begin{array}{l}
  \mathcal{A}_{1,1,1,1}=1,\mathcal{A}_{1,1,1,10}=1,\mathcal{A}_{1,1,10,10}=1,\mathcal{A}_{1,10,10,10}=1,\mathcal{A}_{2,2,2,2}=6, \mathcal{A}_{2,2,2,4}=2,\\
\mathcal{A}_{2,2,2,8}=2,\mathcal{A}_{2,2,2,9}=2,\mathcal{A}_{2,2,4,4}=2,\mathcal{A}_{2,2,4,8}=1, \mathcal{A}_{2,2,4,9}=1,\mathcal{A}_{2,2,8,8}=2,\\
\mathcal{A}_{2,2,8,9}=1,\mathcal{A}_{2,2,9,9}=2,\mathcal{A}_{2,4,4,4}=2, \mathcal{A}_{2,4,4,8}=1,\mathcal{A}_{2,4,4,9}=1,\mathcal{A}_{2,4,8,8}=1,\\
\mathcal{A}_{2,4,8,9}=1,\mathcal{A}_{2,4,9,9}=1, \mathcal{A}_{2,8,8,8}=2,\mathcal{A}_{2,8,8,9}=1,\mathcal{A}_{2,8,9,9}=1,\mathcal{A}_{2,9,9,9}=2,\\
\mathcal{A}_{4,4,4,4}=6, \mathcal{A}_{4,4,4,8}=2,\mathcal{A}_{4,4,4,9}=2,\mathcal{A}_{4,4,8,8}=2,\mathcal{A}_{4,4,8,9}=1,\mathcal{A}_{4,4,9,9}=2, \\
\mathcal{A}_{4,8,8,8}=2,\mathcal{A}_{4,8,8,9}=1,\mathcal{A}_{4,8,9,9}=1,\mathcal{A}_{4,9,9,9}=2,\mathcal{A}_{5,5,5,5}=2,\mathcal{A}_{5,5,5,7}=1, \\
\mathcal{A}_{5,5,5,9}=1,\mathcal{A}_{5,5,7,7}=1,\mathcal{A}_{5,5,7,9}=1,\mathcal{A}_{5,5,9,9}=1, \mathcal{A}_{5,7,7,7}=1,\mathcal{A}_{5,7,7,9}=1,\\
\mathcal{A}_{5,7,9,9}=1,\mathcal{A}_{5,9,9,9}=1,\mathcal{A}_{6,6,6,6}=3, \mathcal{A}_{6,6,6,7}=1,\mathcal{A}_{6,6,6,9}=1,\mathcal{A}_{6,6,6,10}=1,\\
\mathcal{A}_{6,6,7,7}=1,\mathcal{A}_{6,6,7,9}=1,\mathcal{A}_{6,6,9,9}=1,\mathcal{A}_{6,6,10,10}=1,\mathcal{A}_{6,7,7,7}=1,\mathcal{A}_{6,7,7,9}=1,\\
\mathcal{A}_{6,7,9,9}=1,\mathcal{A}_{6,9,9,9}=1,\mathcal{A}_{6,10,10,10}=1,\mathcal{A}_{7,7,7,7}=4,\mathcal{A}_{7,7,7,9}=2,\mathcal{A}_{7,7,9,9}=2,\\
\mathcal{A}_{7,9,9,9}=2,\mathcal{A}_{8,8,8,8}=8,\mathcal{A}_{8,8,8,9}=3,\mathcal{A}_{8,8,8,10}=1,\mathcal{A}_{8,8,9,9}=3,\mathcal{A}_{8,8,9,10}=1,\\
\mathcal{A}_{8,8,10,10}=1, \mathcal{A}_{8,9,9,9}=3,\mathcal{A}_{8,9,9,10}=1,\mathcal{A}_{8,9,10,10}=1,\mathcal{A}_{8,10,10,10}=1,\mathcal{A}_{9,9,9,9}=12, \\ \mathcal{A}_{9,9,9,10}=1,\mathcal{A}_{9,9,10,10}=1,\mathcal{A}_{9,10,10,10}=1,\mathcal{A}_{10,10,10,10}=4,\\
\end{array}
$$
and the other entries are zero, except permutations of the above indices.
$\mathcal{A}$ is a strongly symmetric hierarchically dominated nonnegative tensor,
so it is completely positive \cite{Qi}.
A nonnegative decomposition of length 20 of $\mathcal{A}$ is also given in \cite{Qi}.

We apply Algorithm \ref{Algorithm} and choose $d=6$ and $k=3$ in Step 0.
It terminates at Step 3 with $k=3$,
and gives the following nonnegative decomposition
\[
\mathcal{A} =\sum_{i=1}^{20} \rho_i (v^i)^{\otimes 4},
\]
where $\rho_i$ and $v^i$ are:
{\footnotesize
\begin{align*}
 & \rho_1=1.0000,\ v^1=(0.0000,
    0.0000,
    0.0000,
    0.0000,
    0.0000,
    1.0000,
    0.0000,
    0.0000,
    0.0000,
    0.0000)^T, \\
 &\rho_2=2.0000,\ v^2=(0.0000,
    0.0000,
    0.0000,
    1.0000,
    0.0000,
    0.0000,
    0.0000,
    0.0000,
    0.0000,
    0.0000)^T,  \\
 &\rho_3=2.0000,\ v^3=(0.0000,
    0.0000,
    0.0000,
    0.0000,
    0.0000,
    0.0000,
    1.0000,
    0.0000,
    0.0000,
    0.0000)^T,  \\
 &\rho_4=2.0000,\ v^4=(0.0000,
    1.0000,
    0.0000,
    0.0000,
    0.0000,
    0.0000,
    0.0000,
    0.0000,
    0.0000,
    0.0000)^T,  \\
 &\rho_5=1.0000,\ v^5=(0.0000,
    0.0000,
    0.0000,
    0.0000,
    1.0000,
    0.0000,
    0.0000,
    0.0000,
    0.0000,
    0.0000)^T, \\
 &\rho_6=4.0000,\ v^6=(0.0000,
    0.7071,
    0.0000,
    0.7071,
    0.0000,
    0.0000,
    0.0000,
    0.0000,
    0.0000,
    0.0000)^T, \\
&\rho_7=4.0000,\ v^7=(0.0000,
    0.0000,
    0.0000,
    0.0000,
    0.0000,
    0.7071,
    0.0000,
    0.0000,
    0.0000,
    0.7071)^T, \\
 &\rho_8=4.0000,\ v^8=(0.0000,
    0.0000,
    0.0000,
    0.7071,
    0.0000,
    0.0000,
    0.0000,
    0.7071,
    0.0000,
    0.0000)^T, \\
 &\rho_9=1.0000,\ v^9=(0.0000,
    0.0000,
    0.0000,
    0.0000,
    0.0000,
    0.0000,
    0.0000,
    0.0000,
    0.0000,
    1.0000)^T, \\
 &\rho_{10}=3.0000,\ v^{10}=(0.0000,
    0.0000,
    0.0000,
    0.0000,
    0.0000,
    0.0000,
    0.0000,
    1.0000,
    0.0000,
    0.0000)^T, \\
 &\rho_{11}=4.0000,\ v^{11}=(0.7071,
    0.0000,
    0.0000,
    0.0000,
    0.0000,
    0.0000,
    0.0000,
    0.0000,
    0.0000,
    0.7071)^T,  \\
 &\rho_{12}=4.0000,\ v^{12}=(0.0000,
    0.0000,
    0.0000,
    0.7071,
    0.0000,
    0.0000,
    0.0000,
    0.0000,
    0.7071,
    0.0000)^T, \\
 &\rho_{13}=4.0000,\ v^{13}=(0.0000,
    0.7071,
    0.0000,
    0.0000,
    0.0000,
    0.0000,
    0.0000,
    0.7071,
    0.0000,
    0.0000)^T, \\
&\rho_{14}=9.0000,\ v^{14}=(0.0000,
    0.0000,
    0.0000,
    0.0000,
    0.0000,
    0.5774,
    0.5774,
    0.0000,
    0.5774,
    0.0000)^T, \\
 &\rho_{15}=5.0000,\ v^{15}=(0.0000,
    0.0000,
    0.0000,
    0.0000,
    0.0000,
    0.0000,
    0.0000,
    0.0000,
    1.0000,
    0.0000)^T, \\
 &\rho_{16}=4.0000,\ v^{16}=(0.0000,
    0.7071,
    0.0000,
    0.0000,
    0.0000,
    0.0000,
    0.0000,
    0.0000,
    0.7071,
    0.0000)^T, \\
 &\rho_{17}=9.0000,\ v^{17}=(0.0000,
    0.0000,
    0.0000,
    0.0000,
    0.5774,
    0.0000,
    0.5774,
    0.0000,
    0.5774,
    0.0000)^T, \\
 &\rho_{18}=4.0000,\ v^{18}=(0.0000,
    0.0000,
    0.0000,
    0.0000,
    0.0000,
    0.0000,
    0.0000,
    0.7071,
    0.7071,
    0.0000)^T, \\
 &\rho_{19}=16.0000,\ v^{19}=(0.0000
    0.5000,
    0.0000,
    0.5000,
    0.0000,
    0.0000,
    0.0000,
    0.5000,
    0.5000,
    0.0000)^T, \\
 &\rho_{20}=9.0000,\ v^{20}=(0.0000,
    0.0000,
    0.0000,
    0.0000,
    0.0000,
    0.0000,
    0.0000,
    0.5774,
    0.5774,
    0.5774)^T.
\end{align*}
}
In fact, the decomposition we get above is the same as that given in \cite{Qi}.
\eex

\bex\label{Exampler} \upshape
Consider the tensor $\mathcal{A}\in  \mathrm{S}^4(\mathrm{R}^{3})$ given as:
\begin{equation}\label{Exar}
\mathcal{A} = \sum^8_{k=1}
(u^k)^{\otimes 4},
\end{equation}
where $u^k$ is randomly generated:
\begin{align*}
  ( u^1 ,  &u^2 ,  u^3 ,  u^4 ,  u^5 ,  u^6 ,  u^7,  u^8 )\\
    =&    \left(
  \begin{array}{cccccccc}
    0.9863  &  0.8272  &  0.5859  &  0.0372  &  0.7544  &  0.1114  &  0.8892  &  0.2140 \\
    0.0520  &  0.2332  &  0.6013  &  0.4021 &   0.4401  &  0.0895   & 0.1871  &  0.1018 \\
    0.9367   & 0.6805  &  0.2352   & 0.8218  &  0.2196  &  0.1421   & 0.5704  &  0.1826 \\
  \end{array}
\right).
\end{align*}

Obviously, $\mathcal{A}$ is completely positive.
We apply Algorithm \ref{Algorithm} to verify this fact.
We choose $d=6$ and $k=3$ in Step 0. Algorithm \ref{Algorithm} terminates at Step 3 with $k=4$, and gives the nonnegative decomposition
\[
\mathcal{A} = \sum_{i=1}^{6}  \rho_i (v^i)^{\otimes 3},
\]
where $\rho_i$ and $v^i$ are:
{\small
 \begin{align*}
 & \rho_1=0.5821,\ v^1=(0.6929,
    0.6921,
    0.2024)^T, \\
 &\rho_2=1.5228,\ v^2=(0.8402,
    0.3689,
    0.3975)^T,  \\
 &\rho_3=5.1374,\ v^3=(0.7443,
    0.0702,
    0.6642)^T,  \\
 &\rho_4=0.1409,\ v^4=(0.5335,
    0.1934,
    0.8234)^T,  \\
 &\rho_5=0.1879,\ v^5=(0.3474,
    0.5526,
    0.7576)^T, \\
 &\rho_6=0.5794,\ v^6=(0.0000,
    0.4263,
    0.9046)^T.
\end{align*}
}

Algorithm \ref{Algorithm} gives a shorter nonnegative decomposition than that given in \eqref{Exar}.

\eex

\bex\label{Exampler54} \upshape
Consider the tensor $\mathcal{A}\in  \mathrm{S}^5(\mathrm{R}^{4})$ given as:
\begin{equation}\label{Exar54}
\mathcal{A} = \sum^{12}_{k=1}
(u^k)^{\otimes 5},
\end{equation}
where $u^k$ is randomly generated:
{\small
\begin{align*}
 & (u^1,  u^2,  u^3,  u^4,  u^5,  u^6)
= \left(
  \begin{array}{cccccc}
        0 &   0.9323   &      0  &  0.5363    &     0    &       0 \\
        0  &  2.4751  &  0.9494  &       0   &      0    &       0 \\
        0  &       0  &  1.7464   & 0.4275   &      0    &  0.0212 \\
   0.0706   & 1.0381  &  1.0540    &     0   & 0.1086     &      0 \\
  \end{array}
\right), \\
& (u^7,  u^8,  u^9,  u^{10},  u^{11},  u^{12})
=  \left(
  \begin{array}{cccccc}
               0   & 0.6436   & 1.2125  &  0.2525  &     0  &  0.5494\\
          0.6237   & 1.4777    &     0  &       0 & 0.4431   & 0.9135\\
          1.1552   &      0     &    0  &  1.1050  &    0    &     0\\
               0   &      0      &   0   & 0.3298    &  0   &      0\\
  \end{array}
\right).
\end{align*}
}

Clearly, $\mathcal{A}$ is completely positive. We apply Algorithm \ref{Algorithm} and choose $d=6$ and $k=3$ in Step 0. It terminates at Step 3  with $k=3$,  and gives the nonnegative decomposition
\[
\mathcal{A} = \sum_{i=1}^{8}  \rho_i (v^i)^{\otimes 5},
\]
where $\rho_i$ and $v^i$ are:
{\small
 \begin{align*}
 & \rho_1=7.7107,\ v^1=(0.3820,
    0.9242,
    0.0000,
    0.0000)^T, \\
 &\rho_2=4.5500,\ v^2=(0.4619,
    0.8869,
    0.0000,
    0.0000)^T,  \\
 &\rho_3=2.6233,\ v^3=(1.0000,
    0.0004,
    0.0000,
    0.0000)^T,  \\
 &\rho_4=185.1917,\ v^4=(0.3281,
    0.8711,
    0.0000,
    0.3653)^T,  \\
 &\rho_5=3.9090,\ v^5=(0.0000,
    0.4751,
    0.8799,
    0.0000)^T, \\
 &\rho_6=0.1508,\ v^6=(0.7819,
    0.0000,
    0.6233,
    0.0000)^T, \\
&\rho_7=57.6559,\ v^7=(0.0000,
    0.4220,
    0.7761,
    0.4685)^T, \\
 &\rho_8=2.2945,\ v^8=(0.2139,
    0.0000,
    0.9361,
    0.2794).
\end{align*}
}

The length of the nonnegative decomposition above is shorter than that of \eqref{Exar54}.

\eex

\bex\label{Example6} \upshape
Consider the tensor $\mathcal{A}\in  \mathrm{S}^5(\mathrm{R}^{8})$ given as:
\begin{equation}\label{Exa6}
\mathcal{A} = \sum^5_{k=1}
(u^k)^{\otimes 5},
\end{equation}
where $u^k$ is randomly generated:
\[
(   u^1 ,  u^2 ,  u^3 ,  u^4 ,  u^5)=
 \left(
  \begin{array}{ccccc}
    0.7022 &   0.1176  &  0.3839  &  0.7534  &  0.3798 \\
    0.4504  &  0.8497 &   0.8385 &   0.0358 &   0.3089\\
    0.4799   & 0.2428  &  0.1933 &   0.9760  &  0.8814\\
    0.9176  &  0.6226  &  0.3218 &   0.9913  &  0.8675\\
    0.0614  &  0.9455  &  0.8484  &  0.5675  &  0.3120\\
    0.7387 &   0.1804  &  0.2369   & 0.5012 &   0.7118\\
    0.1235   & 0.8027  &  0.0264  &  0.4344 &   0.2808\\
    0.5151  &  0.2090  &  0.7212  &  0.4743  &  0.9899\\
  \end{array}
\right).
\]

We apply Algorithm \ref{Algorithm} and choose $d=6$ and $k=3$ in Step 0.
It terminates at Step 3 with $k=3$,
and gives the nonnegative decomposition
\[
\mathcal{A} =\sum_{i=1}^{5} \rho_i (v^i)^{\otimes 5},
\]
where $\rho_i$ and $v^i$ are:
{\small
 \begin{align*}
 & \rho_1=21.8169,\ v^1=(0.2050
    0.1667
    0.4758
    0.4683
    0.1684
    0.3842
    0.1516
    0.5343)^T, \\
 &\rho_2=10.8812,\ v^2=(0.4356
    0.2794
    0.2977
    0.5693
    0.0381
    0.4583
    0.0766
    0.3196)^T,  \\
 &\rho_3=22.7706,\ v^3=(0.4032
    0.0192
    0.5224
    0.5305
    0.3037
    0.2682
    0.2325
    0.2539)^T,  \\
 &\rho_4=7.9195,\ v^4=(0.2538
    0.5544
    0.1278
    0.2127
    0.5608
    0.1566
    0.0175
    0.4768)^T,  \\
 &\rho_5=13.0825,\ v^5=(0.0703
    0.5081
    0.1452
    0.3723
    0.5654
    0.1079
    0.4799
    0.1250)^T.
\end{align*}
}

In fact, the nonnegative decomposition given above is the same as \eqref{Exa6}.
\eex

\bex\label{Example7} \upshape
Consider the tensor $\mathcal{A}\in  \mathrm{S}^4(\mathrm{R}^{10})$ given as:
\begin{equation}\label{Exa7}
\mathcal{A} = \sum^5_{k=1}
(u^k)^{\otimes 4},
\end{equation}
where $u^k$ is randomly generated:
\[
(    u^1 ,  u^2 ,  u^3 ,  u^4 ,  u^5)=
 \left(
  \begin{array}{ccccc}
         0  &  0.4764   &      0   &      0  &  0.2689 \\
    0.4538   & 2.1852   & 0.0659  &       0   & 0.2750\\
         0  &  0.6909   & 0.3847  &  0.8428   & 0.3589\\
    0.5668  &       0   & 0.8943 &   0.4031   &      0\\
    0.6059  &       0   & 0.7172 &        0   &      0\\
         0 &   0.6203   &      0 &   1.0202  &  0.2052\\
    0.1415 &   0.7685  &  0.5979 &   0.4848   &      0\\
         0 &   1.3676   & 0.0973 &   0.3973  &       0\\
    0.2291 &   0.5564 &        0 &   0.2032   &      0\\
         0  &       0    &     0  &  0.9733   &      0\\
  \end{array}
\right).
\]

We apply Algorithm \ref{Algorithm} and choose $d=6$ and $k=3$ in Step 0.
It terminates at Step 3 with $k=3$,
and gives the nonnegative decomposition
\[
\mathcal{A} =\sum_{i=1}^{5} \rho_i (v^i)^{\otimes 4},
\]
where $\rho_i$ and $v^i$ are:
{\footnotesize
 \begin{align*}
 & \rho_1=0.1017,\ v^1=(0.4762,
    0.4870,
    0.6356,
    0.0000,
    0.0000,
    0.3634,
    0.0000,
    0.0000,
    0.0000,
    0.0000)^T, \\
 &\rho_2=74.5574,\ v^2=(0.1621,
    0.7437,
    0.2351,
    0.0000,
    0.0000,
    0.2111,
    0.2615,
    0.4654,
    0.1893,
    0.0000)^T,  \\
 &\rho_3=0.9347,\ v^3=(0.0000,
    0.4616,
    0.0000,
    0.5764,
    0.6162,
    0.0000,
    0.1439,
    0.0000,
    0.2330,
    0.0000)^T,  \\
 &\rho_4=3.3617,\ v^4=(0.0000,
    0.0487,
    0.2841,
    0.6605,
    0.5296,
    0.0000,
    0.4416,
    0.0719,
    0.0000,
    0.0000)^T,  \\
 &\rho_5=10.8577,\ v^5=(0.0000,
    0.0000,
    0.4643,
    0.2221,
    0.0000,
    0.5620,
    0.2671,
    0.2189,
    0.1120,
    0.5362)^T.
\end{align*}
}

The nonnegative decomposition given above is the same as \eqref{Exa7}.
\eex

\section{Conclusions and discussions}
We consider the completely positive tensor and decomposition problem.
We formulate it as an $E$-truncated $K$-moment problem
and present a semidefinite algorithm (Algorithm \ref{Algorithm}) for it.
If the tensor is not completely positive, Algorithm \ref{Algorithm} can give a certificate for it.
If it is completely positive, Algorithm \ref{Algorithm} can give a nonnegative decomposition of it, without assuming the length of the decomposition is known.
Numerical experiments show that Algorithm \ref{Algorithm} is efficient in solving the general completely positive tensor decomposition problem.

For a tensor that is not completely positive,
an interesting problem is how to find a best fit of it.
This will be our future work.

\end{document}